\documentclass{amsart}
\usepackage{amssymb,amsthm,amsmath,epsfig,latexsym,xypic,supertabular}
\usepackage{calc}
\usepackage{latexsym}
\usepackage[arrow,matrix,graph,frame,poly,arc,tips]{xy}
\usepackage{paralist}

\begin{document}

\newcommand{\mmbox}[1]{\mbox{${#1}$}}
\newcommand{\affine}[1]{\mmbox{{\mathbb A}^{#1}}}
\newcommand{\Ann}[1]{\mmbox{{\rm Ann}({#1})}}
\newcommand{\caps}[3]{\mmbox{{#1}_{#2} \cap \ldots \cap {#1}_{#3}}}
\newcommand{\N}{{\mathbb N}}
\newcommand{\Z}{{\mathbb Z}}
\newcommand{\R}{{\mathbb R}}
\newcommand{\KK}{{\mathbb K}}
\newcommand{\A}{{\mathcal A}}
\newcommand{\B}{{\mathcal B}}
\newcommand{\OO}{{\mathcal O}}
\newcommand{\C}{{\mathbb C}}
\newcommand{\PP}{{\mathbb P}}

\newcommand{\Tor}{\mathop{\rm Tor}\nolimits}
\newcommand{\Ext}{\mathop{\rm Ext}\nolimits}
\newcommand{\Hom}{\mathop{\rm Hom}\nolimits}
\newcommand{\im}{\mathop{\rm Im}\nolimits}
\newcommand{\rk}{\mathop{\rm rk}\nolimits}
\newcommand{\codim}{\mathop{\rm codim}\nolimits}
\newcommand{\supp}{\mathop{\rm supp}\nolimits}
\newcommand{\coker}{\mathop{\rm coker}\nolimits}
\newcommand{\height}{\mathop{\rm ht}\nolimits}
\newcommand{\rank}{\mathop{\rm rank}\nolimits}
\newcommand{\Ker}{\mathop{\rm Ker}\nolimits}
\newcommand{\depth}{\mathop{\rm depth}\nolimits}
\newcommand{\pd}{\mathop{\rm pd}\nolimits}
\newcommand{\grade}{\mathop{\rm grade}\nolimits}
\newcommand{\ann}{\mathop{\rm Ann }\nolimits}
\newcommand{\ra}{\rightarrow}
\newcommand{\lra}{\longrightarrow}
\newcommand{\ass}{\mathop{\rm Ass }\nolimits}
\newcommand{\Syz}{\mathop{\rm Syz}\nolimits}
\newcommand{\Spec}{\mathop{\rm Spec}\nolimits}

\sloppy
\newtheorem{defn0}{Definition}[section]
\newtheorem{prop0}[defn0]{Proposition}
\newtheorem{conj0}[defn0]{Conjecture}
\newtheorem{thm0}[defn0]{Theorem}
\newtheorem{lem0}[defn0]{Lemma}
\newtheorem{corollary0}[defn0]{Corollary}
\newtheorem{remark0}[defn0]{Remark}
\newtheorem{example0}[defn0]{Example}

\newenvironment{defn}{\begin{defn0}}{\end{defn0}}
\newenvironment{prop}{\begin{prop0}}{\end{prop0}}
\newenvironment{conj}{\begin{conj0}}{\end{conj0}}
\newenvironment{thm}{\begin{thm0}}{\end{thm0}}
\newenvironment{lem}{\begin{lem0}}{\end{lem0}}
\newenvironment{cor}{\begin{corollary0}}{\end{corollary0}}
\newenvironment{rmk}{\begin{remark0}}{\end{remark0}}
\newenvironment{exm}{\begin{example0}\rm}{\end{example0}}

\newcommand{\msp}{\renewcommand{\arraystretch}{.5}}
\newcommand{\rsp}{\renewcommand{\arraystretch}{1}}

\newenvironment{lmatrix}{\renewcommand{\arraystretch}{.5}\small
 \begin{pmatrix}} {\end{pmatrix}\renewcommand{\arraystretch}{1}}
\newenvironment{llmatrix}{\renewcommand{\arraystretch}{.5}\scriptsize
 \begin{pmatrix}} {\end{pmatrix}\renewcommand{\arraystretch}{1}}
\newenvironment{larray}{\renewcommand{\arraystretch}{.5}\begin{array}}
 {\end{array}\renewcommand{\arraystretch}{1}}

\title
{Syzygy Theorems via Comparison of Order Ideals on a Hypersurface}

\author{Phillip Griffith}
\address{Griffith: Mathematics Department \\ University of
 Illinois \\
   Urbana \\ IL 61801\\USA}
\email{pgriffit@illinois.edu}
\author{Alexandra Seceleanu}
\address{Seceleanu: Mathematics Department \\ University of
 Illinois \\
   Urbana \\ IL 61801\\USA}
\email{asecele2@math.uiuc.edu}


\begin{abstract}
We introduce  a weak order ideal property that suffices for establishing the Evans-Griffith Syzygy Theorem. We study this weak order ideal property in settings that allow for comparison between homological algebra over a local ring $R$ versus a hypersurface ring $\bar{R}=R/(x^n)$. Consequently we solve some relevant cases of the Evans-Griffith syzygy conjecture over  local rings of unramified mixed characteristic $p$, with the case of syzygies of prime ideals of Cohen-Macaulay local rings of unramified mixed characteristic being noted.  We reduce the remaining considerations to modules annihilated by $p^s$, $s>0$, that have finite projective dimension over a hypersurface ring.
\end{abstract}

\maketitle

\section{Introduction}

Let $R$ be a commutative ring and let $x\in R$ be a non-zerodivisor. Our main objective is to understand how homological algebra over a ring $\bar{R}=R/(x)$ relates to that over $R$, with a focus on establishing relations between order ideals of syzygies over the two rings. As explained in the final remarks, it is particularly interesting to understand the case when both $R$ and $\bar{R}$ are regular, with $R$ of mixed characteristic and  $\bar{R}$ ramified or of characteristic $p$. 
\smallskip

 For $k \geq 0$, we say that an $R$-module $E$ is a $k^{th}$ syzygy if it arises as the cokernel of the $(k+1)^{st}$ differential in a  projective acyclic complex of finitely generated $R$-modules. In case $R$ is local, it suffices to consider minimal acyclic free resolutions. We shall often use the notation $\Syz_k(M)$ for the $k^{th}$ syzygy of an $R$-module $M$. When $M$
has finite projective dimension and, in addition, $M$ has a finite free resolution
 then the rank of $M$ may be defined as the alternating sum of the ranks in any finite free resolution of $M$.\smallskip
 
A celebrated homological theorem that yields a tight lower bound on ranks of syzygies was proved by Evans and Griffith in \cite{EG1}. In its most general form their theorem states:

\begin{thm}[Syzygy Theorem]
 A finitely generated and finite projective dimension $k^{th}$ syzygy module over a local ring containing a field, if not free, has rank at least $k$.
\end{thm}

Evans and Griffith gave the original proof of the Syzygy Theorem in \cite{EG1} in the case of a local integral domain containing a field. Several different styles of proofs and generalizations have since appeared, the most notable of these being the characteristic $p$ proof of Hochster and Huneke \cite{HH} and the generalization by Bruns \cite{B} in which the minimal free complex is allowed some positive homology and the domain condition is dropped. Moreover, work of Ogoma \cite{O}, Dutta \cite{D} and Hochster \cite{H} showed that the Syzygy Theorem can be deduced from the Improved New Intersection Theorem and that the latter is equivalent to the Canonical Element Conjecture respectively. Further progress was made to prove the analogous result in the graded case in any characteristic in \cite{EG2} and for mixed characteristic regular local rings of dimension at most five in \cite{DG}. In the local mixed characteristic context the Syzygy Problem is open and we aim here at applying our methods to settle a few relevant cases. 

\begin{defn}
If $E$ is an $R$-module and $e \in E$,  there is an induced  $R$-homomorphism $e : \Hom_R(E,R)\ra R$ defined by $e(f) = f(e)$ the image of which is the order ideal
$$O_E(e) = \{f(e) | f \in \Hom_R(E,R)=E^*\}.$$
\end{defn}

The central idea in \cite{EG1}, \cite{EG2}, \cite{DG} and key to the proof of  the Syzygy Theorem was to establish first a stronger result, namely the Order Ideal Theorem for $k^{th}$ syzygies of finite projective dimension. For such a syzygy $E$, this theorem states that $\grade O_E(e)\geq k$, for every minimal generator $e\in E- mE$. For application to establishing the Syzygy Theorem, it was observed in \cite{EG2} in the homogeneous setting that having at least one minimal generator satisfying the grade inequality of the Order Ideal Theorem would suffice.  In this article, we introduce and study this lesser condition under the name of weak order ideal property for $k^{th}$ syzygies:

\begin{defn}
Let $R$ be a Noetherian local ring and let $E$ be a $k^{th}$ syzygy $R$-module. We say $E$ satisfies the weak order ideal property if there exists $e\in E- mE$ such that $\grade O_E(e)\geq k$ (or equivalently over an $S_k$ ring $R$, if $\height_R O_E(e)\geq k$).
\end{defn}

The paper is structured as follows: we begin by establishing several reductions that can be made for the rank of syzygies problem in mixed characteristic. In particular we reduce to syzygies of modules annihilated by powers of $p$ and establish, in the spirit of Bruns' three-generated ideal theorem, some special classes of three generated ideals that are of interest (Theorem \ref{reductionAnnp}). We further reduce under mild hypotheses to syzygies of finite projective dimension modules over a hypersurface ring $S=R/(p^n)$ (Proposition \ref{redfinitepd}).  In section three we prove that the order ideal property implies the Syzygy Theorem in this more subtle setting of modules over hypersurface rings  (Proposition \ref{sufficient}). In the fourth and fifth sections we develop the machinery for comparison of order ideals of syzygy modules over $R$ versus $R/(x)$ or $R/(x^n)$.  Our approach here is to study situations in which we can achieve the weak order ideal property for $R$-syzygies given that the Order Ideal Theorem holds over the respective hypersurface ring. Let $E$ represent a $k^{th}$ syzygy module over $R$ and let $E'$ represent the same over  $R/(x)$. A useful comparison arises whenever there is a homomorphism $E\lra E'$ that remains nontrivial upon tensoring with the residue field. The main results in Theorems \ref{comp} and \ref{2ndComp} describe two situations when such a conclusion can be achieved. The sixth section contains results on stronger bounds on ranks of syzygies which can be deduced for modules annihilated by $p$ and also under the incomparable hypothesis that the module being resolved is weakly liftable in the sense of Auslander, Ding and Solberg \cite{ADS}.


A characterization of syzygies can be given in terms of Serre's property $S_k$: we say that an $R$-module $M$ satisfies the Serre condition $S_k$ if for each prime ideal $\frak{p}\in \Spec R, \depth_{R_{\frak{p}} }M_{\frak{p}}\geq min(k, \dim R_{ \frak{p}}).$ Over $S_k$ rings, $k^{th}$ syzygies can simply be characterized as modules satisfying property $S_k$. Although we can state many of our results over $S_k$ domains, with respect to the order ideal techniques employed in \cite{EG}, \cite{EG2} only the unramified regular local case is most likely to yield positive results. The final section contains details on the importance of understanding the case where  $R$ is assumed to be regular and unramified.

\section{Reductions}
\subsection{Reduction to modules annihilated by powers of $p$}

Throughout this section $R$ is a  local unique factorization domain of mixed characteristic $p$. In this context, any possible counterexample to the Syzygy Conjecture can only occur for $k^{th}$ syzygies $E$ with $k\geq 3$, since second syzygies of rank one must be  isomorphic to $R$.  Therefore we may deform the initial two terms of any free  resolution as long as the third and higher syzygies remain the same. Moreover, an affirmative answer to the Syzygy Theorem in equal characteristic implies that $E[p^{-1}]$ will be $R[p^{-1}]$ projective for any counterexample $E$.

We recall the universal pushforward construction (page 49 in \cite{EG}) 

\begin{prop}\label{universalpushforward}
 Let $E$ be a $k^{th}$ syzygy over $R$. Then one can construct an exact sequence of length $k$ of free $R$-modules  $$0\ra E\ra R^{n_k}\ra \ldots \ra R^{n_1}$$ called the universal pushforward of $E$. 
If we further assume that $E[p^{-1}]$ is $R[p^{-1}]$-projective, then all the syzygies of the universal pushforward sequence become projective upon inverting $p$.
\end{prop}
\begin{proof}
 Let ${f_1,\ldots ,f_{n_k}}$ generate $E^*$ and map $E\stackrel{u}\ra R^{n_k}$ via setting $u(m)=(f_1(m),\ldots,f_{n_k}(m))$. If $E$ is a $k^{th}$ syzygy ($k\geq 1$), this is a monomorphism which gives rise to a dual exact sequence 
$0\ra E\ra R^{n_k}\ra C\ra 0$, with $C$ a $(k-1)^{st}$ syzygy. If we further assume that $E[p^{-1}]$ is $R[p^{-1}]$-projective, this short exact sequence must become locally split upon inverting $p$, yielding that $C[p^{-1}]$ is also $R[p^{-1}]$-projective. Now  we may repeat the process as long as $C$ is at least a first syzygy (i.e. $k$ times) to obtain a long exact sequence of length $k$ in which all the syzygies become projective upon inverting $p$.
\end{proof}

We further recall the statement of the Bourbaki Theorem (Theorem 2.14 in \cite{EG}).
\begin{thm}[Bourbaki]\label{Bourbaki}
 Let $R$ be a normal domain and let $N$ be a finitely generated torsion-free $R$-module. Then there exists a free submodule $F$ of $N$ such that $N/F$ is isomorphic to an ideal.
\end{thm}

We begin with a result allowing us to remove certain associated primes.

\begin{prop}\label{propremove}
 Let $R$ be a local unique factorization domain of  mixed characteristic $p$ such that $R/(p)$ is still a unique factorization domain and suppose $E$ is a $k^{th}$ syzygy of an $R$-module $M$ where $k\geq3$. If $P\in Ass_R M$ with $p\in P$ and $\height P\leq 2$, then $E$ is also a $k^{th}$ syzygy for $M'$ where 
$$0\ra R/P\ra M\ra M'\ra 0 \mbox{ is exact.}$$
\end{prop}
\begin{proof}
 We shall discuss only the height 2 case since height one requires similar arguments. We note that $P$ is necessarily of the form $P=(p,q)$ where the class of $q$ represents a prime in $\bar{R}$. We may view a  free resolution of $M$ being formed  via the horseshoe lemma as a direct sum of the minimal free resolution for $M'$ and the length 2 resolution for $R/P$. Therefore the third and higher syzygies of $M$ are preserved in the resolution of $M'$.

\end{proof}

\begin{thm}\label{reductionAnnp}
 Let $R$ be a  local normal unique factorization domain of mixed characteristic $p$ such that $R/(p)$ is still a unique factorization domain. Suppose $E$ is a $k^{th}$ syzygy ($k\geq 3$) in a finite free resolution $\mathbf{F}.\ra M$  and $E[p^{-1}]$ is a projective $R[p^{-1}]$-module. Then we may assume that $M$ is any one of the following types:
\begin{enumerate}
 \item the module $M$ is annihilated by $p^s$ for some $s>0$ and $\height_R(ann_R M)\geq 3$;
 \item the module $M\simeq R/I$ where the ideal $I$ has the property that for $P\in Ass_R(R/I)$ and $p\notin P$, $ht P\leq 2$.
\item If furthermore $\Syz_2 M=N$ has the properties $\rank_R N=2$ and $N[p^{-1}]$ is $R[p^{-1}]$-free, then one may take $M$ to be of the form $M\simeq R/(p^s,a,b)$ for some $s>0$.
\end{enumerate}
\end{thm}
\begin{proof}
(1) Applying the universal pushforward construction (Proposition \ref{universalpushforward}) and viewing the module $M$ being resolved as the cokerel of the last map of the pushforward complex, we obtain  that $M[p^{-1}]$ is a projective $R[p^{-1}]$-module. In fact, since it has a finite free resolution, $M[p^{-1}]$ is stably free (Proposition 19.16 in \cite{Eis}). Therefore one may augment $M$ by a suitable free  $R$-module (resulting in a corresponding augmentation of  $\mathbf{F}.$) so that $M[p^{-1}]$ is $R[p^{-1}]$-free.  It follows that there exists  a free submodule $F$ of $M$  such that $T=M/F$ is annihilated by $p^s$ for some $s>0$. Next we may replace $M$ by $T$ while preserving all the $k^{th}$ syzygy modules, for $k\geq 3$ and finally we may remove all height one and two associated primes of $\ann_R T$ using the principle embodied in Proposition \ref{propremove}.

(2) In this instance we run the universal pushforward construction until obtaining a first syzygy module $Z$. Here we employ the Bourbaki Theorem (\ref{Bourbaki}) to obtain a short exact sequence
$$0\ra G\ra Z\ra I\ra 0$$ where $G$ is free and $I$ is an ideal having $\height_R I\geq 2$. It follows that $pd_{R[p^{-1}]}(R/I)[p^{-1}]\leq 2$, since $Z[p^{-1}]$ is $R[p^{-1}]$ projective. Thus if $P\in \ass_R R/I$ and $p\not \in P$ then $\dim R_P \leq 2$ and consequently $ht (P)\leq 2$.

(3) Assuming now  $N$ is the second syzygy of $M$ and $N[p^{-1}]$ is $R[p^{-1}]$-free, one sees that $N^*$ contains an element $e_1$ such that $N^*/e_1R$ is $R$-torsion free and the following sequence which maps $1\in R \mapsto e_1\in N^*$ is split exact: 
$$0\ra R\ra N^* \ra J \ra 0$$

Since $J[p^{-1}]\simeq R[p^{-1}]$, one has that $J$ is isomorphic to a height two ideal in $R$ that contains a power of $p$ so  $(p^s,a)\subseteq I$, where $(p^s,a)$ is an $R$-sequence. If we let $e_2, e_3\in N^*$ correspond modulo $e_1R$ to $p^s$ and $-a$ in $I$ respectively, then one obtains a relation in $N^*$  of the form
$$\lambda_1e_1+\lambda_2e_2+\lambda_3e_3=0$$
where $\lambda_2=a,\lambda_3=p^s$ and $\lambda_1=b\in R$. We set $W=Re_1+Re_2+Re_3\subseteq N^*$ and observe that $N^*/W \simeq I/(a,p^s)$, hence $\ass_R N^*/W$ contains no primes of height one. It follows that $W^*\simeq N^{**}\simeq N$ and that $W$ has a free resolution
\begin{equation}\label{res}
 0\ra R\ra R^3 \ra W \ra 0
\end{equation}
in which $1\in R$ is sent to an element of $R^3$ of the form $\langle b,a,p^s \rangle$. Thus we may dualize the short exact sequence (\ref{res}) and obtain
$$0\ra W^*\ra R^3\ra R \ra R/(a,b,p^s)\ra 0$$
where $\Ext_R^1(W,R)\simeq R/(a,b,p^s)$. Hence we may continue the free resolution from the second syzygy ($N=W^*$) onward as desired.
\end{proof}

\subsection{Reduction to finite projective dimension over a hypersurface ring}
\smallskip
Working in slightly greater generality  we let $x$ be a non-zerodivisor on $R$ and we let $S=R/(x^n)$ and $T$ be an $R$ module such that $x^nT=0$. Clearly one may view the $R$-module $T$ as an $S=R/(x^n)$-module. The first goal of this section is to show that  we may assume $pd_S T<\infty$ for the purpose of examining the $R$-syzygies of $T$.


The main tool we use in the following is the Auslander-Bridger approximation theorem:

\begin{thm}(Auslander-Bridger, Corollary 5.3 in \cite{EG})\label{AusBr}
 Let $R$ be a Gorenstein local ring and let $M$ be a finitely generated reflexive $R$-module. Then there is a free $R$-module $L$ of finite rank and a short exact sequence
$$0\lra U\lra L\bigoplus M \lra M' \lra 0$$
satisfying
\begin{enumerate}
 \item the sequence is dual exact;
\item $U$ is maximal Cohen-Macaulay;
\item $M'$ has finite projective dimension ;
\item the natural map $Ext^i(M',R)\ra Ext^i(M,R)$ is an isomorphism for $i\geq 1$.
\end{enumerate}
\end{thm}

In the following we employ the Auslander-Bridger approximation theorem to reduce to the case of finite resolutions over a hypersurface ring $S$.

\begin{prop}
 Let $R$ be a local ring and $S=R/(x^n)$. If $0\ra M'\ra M\ra M" \ra 0$ is an exact sequence of reflexive $S$-modules that is in addition dual exact, then there is a short exact sequence of $R$-syzygy modules for $M',M$ and $M"$ respectively that is in turn dual exact.
\end{prop}
\begin{proof}
 Consider the following commutative diagram of $R$-modules with exact rows and columns and for which the middle row represents a split exact sequence of free $R$-modules. 
$$
\xymatrixrowsep{20pt}
\xymatrixcolsep{25pt}
\xymatrix{
 0\ar[r]&  Z' \ar[r]\ar[d] & Z \ar[r]\ar[d] & Z'' \ar[r]\ar[d] &  0.\\
0\ar[r]&  F' \ar[r]\ar[d] & F \ar[r]\ar[d] & F'' \ar[r]\ar[d] &  0.\\
0\ar[r]&  M' \ar[r] & M \ar[r] & M'' \ar[r] &  0.\\
}
$$
Next recall that if $M$ is an $x$-torsion module then $Ext_R^1(M,R)\simeq Hom_S(M,S)$ which we denote by $M^{+}$. Applying the functor $Hom_R(\cdot,R)$ to the above diagram one obtains a new diagram in which all columns and all but possibly the middle row are short exact. 
$$
\xymatrixrowsep{20pt}
\xymatrixcolsep{25pt}
\xymatrix{
 0\ar[r]&  (F'')^* \ar[r]\ar[d] & F^* \ar[r]\ar[d] & (F')^* \ar[r]\ar[d] &  0.\\
0\ar[r]&  (Z'')^* \ar[r]\ar[d] & Z^* \ar[r]\ar[d] & (Z')^* \ar[r]\ar[d] &  0.\\
0\ar[r]&  (M'')^+ \ar[r] & M^+ \ar[r] & (M')^+ \ar[r] &  0.\\
}
$$
Finally we apply the nine lemma to see that the middle row is also short exact.
\end{proof}

\begin{cor}
 If in addition to the hypotheses of the above proposition we have that $pd_RM'\leq 1$ (i.e the module $Z'$ is $R$-free), then  the short exact sequence of syzygies $0\ra Z'\ra Z \ra Z''\ra 0$ is split exact.
\end{cor}

\begin{prop}\label{redfinitepd}
 Let $R$ be a local unique factorization domain, $x\in R$ a non-zerodivisor and assume $S=R/(x^n)$ is Gorenstein.  Let $M$ be a reflexive $S$-module and let $0\ra U\ra L\bigoplus M\ra M'\ra 0$ represent the Auslander-Bridger approximation sequence for $M$. Then the first $R$-syzygy modules for $M$ and $M'$ respectively are stably isomorphic and $\Syz_j^R(M)\simeq \Syz_{j}^R(M')$ for $j\geq 2$.
\end{prop}
\begin{proof}
 Noting that $pd_RU\leq 1$ and applying the preceding corollary one obtains that the first $R$-syzygy modules for $M$ and $M'$ respectively are stably isomorphic. Furthermore the non-free parts of these syzygies are isomorphic as one has a cancellation theorem for free direct summands over $R$, since $R\simeq End_R R$ is a local ring. Thus one can find  isomorphisms $\Syz_j^R(M)\simeq \Syz_{j}^R(M')$ for $j\geq 2$.
\end{proof}



This observation allows us to conclude that in order to study $R$-syzygies of $T$, we may use the technique described in the proposition and corollary to replace $M=Syz_2^R(T)$ by a finite $S$-projective dimension reflexive module $M'$. Consequently one may replace $T$ by any $S$-module $T'$ such that $M'=\Syz_2^R(T')$. Most importantly this yields that such a $T'$ will have finite projective dimension over $S$.

 \section{The weak order ideal property implies the Syzygy Theorem}
  
  The purpose of this section is to establish that the weak order ideal property suffices for proving the Syzygy Theorem in the context of modules over hypersurface rings that we have reduced to. Since this idea was first used in the graded setting in \cite{EG2} our wish is to compare and contrast the way in which the weak order ideal property can be used to deduce the Syzygy Theorem in the two contexts:
  \begin{enumerate}
  \item for syzygies of modules over standard graded rings over a DVR;
  \item for $k^{th}$ syzygies of modules over a local hypersurface ring satisfying $S_k$ .
  \end{enumerate}

  Let $R$ be a ring of one of the two types listed above and $E$ a finitely generated $k^{th}$ syzygy of finite projective dimension over $R$. Whenever there exists an element $e\in E$ such that $\height O_E(e)\geq k$, one obtains a short exact sequence of the form 
  $$0\lra R \lra E \lra E'\lra 0$$
  mapping the unit element $1\in R$ to $e\in E$. It follow easily that $E'$ is a $(k-1)^{st}$ syzygy of finite projective dimension and that $\rank _R E'=\rank_R E-1$.
  
  The contrast between the two situations appears at this point: in the graded case, one may choose $e$ to be homogeneous and so $E'$ is naturally again a graded $R$-module and thus the usual induction on rank applies. We now show that the situation in (2) turns out to be slightly more subtle.  
  
   \begin{prop}\label{sufficient}
 Let $E$ be a finitely generated and finite projective dimension $k^{th}$ $R$-syzygy  of a
module over a hypersurface ring  of the form $S=R/(p^s)$, where $R$ is an $S_{k+1}$ ring of mixed characteristic $p$. If $E$ satisfies the weak order ideal property then $\rank_R E\geq k$.
 \end{prop}
 \begin{proof}
Let $e$ be a minimal generator of $E$ with $\height_R O_E(e)\geq k$ and construct $E'$ as the cokernel of the map $1\mapsto e$  in the way described above. Should $\rank _R E=k-1$, then $E'$ will be a $(k-1)^{st}$ syzygy of $\rank_R E'=k-2$. The key point now comes into play: the Syzygy Theorem is known to hold locally over $R[p^{-1}]$, thus both $E[p^{-1}]$ and $E'[p^{-1}]$ will be $R[p^{-1}]$-projective due to being syzygies of too small rank. This implies that the short exact sequence   $$0\lra R[p^{-1}] \lra E[p^{-1}] \lra E'[p^{-1}]\lra 0$$ obtained by tensoring the defining sequence of $E'$ with $R[p^{-1}]$ will be split exact. Lifting the splitting of the first map in the sequence to an $R$-module homomorphism proves that $p^t\in O_E(e)$ for some $t>0$, hence all minimal primes of $O_E(e)$ must contain $p$.  In the setting considered, one knows via viewing $E$ as a module over $R/(p)$ and applying the Order Ideal Theorem there that $\height_{\bar{R}}(O_E(e)+(p))/(p) \geq k$. Therefore one concludes that in fact $\height_R O_E(e)\geq k+1$. This stronger conclusion implies via  a standard  $S_k$ property argument  that $E'$ would in fact be a $k^{th}$ syzygy of rank $k-2$, therefore necessarily $R$-free since the bound predicted by the Syzygy Theorem can fail by at most one (i.e ranks of $k^{th}$ syzygies are known to be at least $k-1$ over any local ring, see page 63 in \cite{EG}).
\end{proof}

In the same way, when working with a general hypersurface $x$ one needs that the Order Ideal Theorem holds over $R/(x)$ and locally over $R[x^{-1}]$ to infer that the weak order ideal property implies the Syzygy Theorem.

\begin{rmk}
A partial converse of the above statement holds as well. Let $R$ be any Noetherian ring. If $E$ is a finitely generated finite projective dimension $k^{th}$ $R$-syzygy of rank at least $k$ which is locally free at any prime of height $k-1$ and if all $(k-1)^{st}$ $R$-syzygies satisfy the weak order ideal property, then $E$ satisfies the weak order ideal property as well.
\end {rmk}

To see this note that by Bruns's theorem on basic elements (see Corollary 2.6 in \cite{EG}), there is a minimal generator $e$ of $E$ such that $E'=E/eR$ is a $(k-1)^{st}$ syzygy and  the sequence
$$0\lra R\lra E\lra E'\lra 0$$ splits in codimension $k-1$, which yields  $\height O_E(e)\geq k$.



\section{Weak Order Ideal Theorem via extension splitting}

 In this section, let $R$ be a commutative Noetherian local ring of any characteristic. In view of the reduction in Theorem \ref{reductionAnnp} (1), we shall work under the more general assumption that a power of a regular element  $x\in R$ annihilates certain extensions. In order to ensure that there is a comparison homomorphism  that remains nontrivial upon tensoring with the residue field we shall also need to assume a superficiality condition on the element $x$.

\begin{defn}
 Let $R$ be a ring , $I$ an ideal, $M$ an $R$-module. We say $x\in I$ is a superficial element of $I$ (of order 1) with respect to $M$ if there exists $c\in \N$ such that  
$$(I^{n+1}M:_Mx)\cap I^cM=I^nM, \mbox{for all }n\geq c.$$
We say $y\in I$ is a superficial element of $I$ of order $d$ with respect to $M$ if there exists $c\in \N$ such that  
$$(I^{n+d}M:_Mx)\cap I^cM=I^nM, \mbox{for all }n\geq c.$$ An example of a superficial element of order $d$ is $y=x^d$ with $x$ a superficial element (of order 1).
\end{defn}

 Assuming that $\depth_I M\geq 1$, every superficial element of $I$ with respect to $M$ is a non-zerodivisor. If $(R,m)$ is local and $R/mR$ is an infinite field, then superficial elements of $M$ with respect to the maximal ideal $m$ are abundant, in fact they form a nonempty Zariski open set in $M/mM$. Our  reference for the stated facts about superficial elements is \cite{HS} section 8.5. Henceforth we shall  consider a local ring $R$ and we shall only use superficial elements with respect to the unique maximal ideal of $R$.\smallskip

The main result of this section is a comparison theorem  between the heights of order ideals of consecutive syzygies modulo a hypersurface. In the following we develop the technical preliminaries needed for our comparison theorem.

\begin{lem}\label{techlemma}
 Let $(R,m)$ be a Noetherian local ring and $E$ a finitely generated $R$-module with minimal  presentation $$\epsilon:0\lra Z\stackrel{\iota}\lra F\ra E\ra 0,$$ where $F$ is free and $\iota(Z)\subseteq m^nF$. Suppose there exists a superficial element (of order 1) $x$ of $m$ with respect to $Z$ such that  $x^n\epsilon=0$ ($\epsilon$ is viewed as an element of $Ext^1(E,Z)$) and let $h:Z\ra Z$ be the map defined by multiplication by $x^n$. Then:
\begin{enumerate}
 \item There is a map $f:F\ra Z$ that makes the following diagram commute
$$
\xymatrixrowsep{20pt}
\xymatrixcolsep{25pt}
\xymatrix{
 0 \ar[r] & Z \ar[r]^{\iota}\ar[d]^{h} & F\ar[r]\ar[dl]^{f} & E \ar[r] & 0\\
 	  & Z\\
}
$$
\item the image of $f$ is not contained in $mZ$
\item the map $f$ induces a map $\bar{f}:E/x^nE\ra Z/x^nZ$ 
$$
\xymatrixrowsep{20pt}
\xymatrixcolsep{25pt}
\xymatrix{
 0 \ar[r] & Z \ar[r]^{\iota}\ar[d]^{h} & F\ar[r]\ar[dl]^{f} & E \ar[r]\ar[d] & 0\\
 	  & Z \ar[d]^{\pi}     &                     & E/x^nE \ar[dll]^{\bar{f}}\\
          & Z/x^nZ                &                     &                        \\
}
$$
\item  $Im(\bar{f})\not \subseteq m(Z/x^nZ)$.
\end{enumerate}
\end{lem}

\begin{proof}
(1) Since $x^n\epsilon=0$, the bottom row of the following diagram splits:
$$
\xymatrixrowsep{20pt}
\xymatrixcolsep{25pt}
\xymatrix{
 0 \ar[r] & Z \ar[r]^{\iota}\ar[d]^{h} & F\ar[r]\ar[d]^{\alpha} & E \ar[r]\ar@{=}[d] & 0\\
 0 \ar[r]& Z \ar[r]		& V\ar[r]\ar@/^/[l]^s & E \ar[r]&	0 \\
}
$$
Define $f=s\circ\alpha$, where $s:V\ra Z$ is the splitting map.

(2) Assuming towards a contradiction that the image of $f$ is contained in $mZ$ and under the hypothesis that the image of $\iota$ is contained in $m^nF$, we obtain $Im(h)=Im(\iota \circ f)\subseteq m^{n+1}Z$. Iterating, $Im(h^k)\subseteq m^{k(n+1)}Z$ or equivalently $x^kZ\subseteq m^{k(n+1)}Z,\forall k\in \N$. 

Let $c$ be the integer in the definition of the superficial element. We show by induction on $i$ that $$x^kZ\subseteq  m^{k(n+1)+i}Z, \ \forall k\geq c, \ \forall i\in \N.$$ The base case ($i=0$) is our previous observation that $x^kZ\subseteq m^{k(n+1)}Z,\forall k\in \N$. Fix $i$ and assume $x^{k}Z\subseteq m^{k(n+1)+i}Z, \forall n\geq c$. Rewriting with $k$ replaced by $k+1$, $x^{k+1}Z\subseteq m^{(k+1)(n+1)+i}Z, \forall k\geq c$, hence by using the superficiality of $x$ one obtains $x^{k}Z\subseteq (m^{nk+n+k+i+1}Z:_Z (x)Z)\cap m^{c}Z=m^{nk+n+k+i}Z\subseteq m^{k(n+1)+i}Z$. Therefore the desired containment holds, leading to the conclusion
$$
x^k Z \subseteq \bigcap_{i=0}^{\infty}m^{k(n+1)+i}Z=0.
$$
This is a contradiction since $x$ is a non-zerodivisor on $Z$.

(3) Let $\pi$ be the projection $\pi:Z\ra Z/xZ$. Then $\pi \circ f|_Z=\pi \circ h=0$, therefore $Z$ is contained in the kernel of $\pi \circ f$, which induces a map $F/Z=E\ra Z/xZ$. Furthermore this map factors through $xE$ yielding $\bar{f}:E/xE\ra Z/xZ$. Since $Im(f) \not \subseteq mZ$ it follows that the image of $\bar{f}$ is not contained in $m(Z/xZ)$.

(4) is a direct consequence of (2).
\end{proof}

Note that the hypothesis $\iota(Z)\subseteq m_R^nF$ holds for $E$ a $k^{th}$ syzygy in a minimal free resolution $(\mathbf{F}.,d.)$ with the matrix of $d_{k+1}$ having entries in $m_R^n$, in other words when the order ideal $O_Z(u)$ is contained in $m^n$ for every minimal generator $u$ of $Z$. The hypothesis $x^n\epsilon=0$ deserves a further analysis. It is equivalent to $x^n$ annihilating $\Ext_R^1(E,\cdot)$ as a functor via the diagram
$$
\xymatrixrowsep{20pt}
\xymatrixcolsep{25pt}
\xymatrix{
 0 \ar[r] & Z \ar[r]^{\iota}\ar[d]^{\cdot x^n} & F\ar[r]\ar[d] & E \ar[r]\ar@{=}[d]^{1_E} & 0\\
 0 \ar[r]& Z \ar[r]		& V\ar[r]\ar@/^/[l]^s & E \ar[r]&	0 \\
}
$$

In \cite{EG2} we find the following useful lemma on comparing heights of order ideals related by taking hypersurface sections. The statements of \cite{EG2} Lemma 2 and Corollary 3 are given for $R$ a graded algebra over a DVR in mixed characteristic $p$ and $x=p$, but analogous statements hold by the same argument more generally as stated below.
 
\begin{lem}\label{EG2lemma}
 Let $(R,m)$ be a local ring which satisfies $S_k$. Let $E$ be a $k^{th}$ syzygy of finite projective dimension. Let $x\in m$, $e\in E- xE$ and set $\bar{R}=R/(x)$, $\bar{E}=E/xE$, $\bar{e}=\mbox{ image of e in }\bar{E}$. Then 
  $$\height_R(O_E(e))\geq \height _R(O_E(e)+(x)/(x))\geq min (k,\height_{\bar{R}}(O_{\bar{E}}(\bar{e}))).$$
Moreover, if $x$ belongs to a minimal associated prime of $O_E(e)$, then $$\height_R(O_E(e))\geq 1+min (k,\height_{\bar{R}}(O_{\bar{E}}(\bar{e}))).$$
\end{lem}

The following result establishes a Weak Order Ideal Theorem.

\begin{thm}[First Weak Order Ideal Theorem]\label{comp} 
Let $(R,m)$ be a local ring satisfying Serre's  property $S_k$. Consider a short exact sequence $0\lra Z\lra F\ra E\ra 0$ with $F$ free. Assume that there exists $x\in m$ with the following properties 
\begin{enumerate}
 \item $x$ is superficial for $m$ with respect to $Z$;
\item $x^n\Ext_R^1(E,Z)=0$ for some integer $n$ with $Z\subseteq m^nF$;
\item $\height_{\bar{R}} O_{\bar{Z}}(\bar{u}) \geq k$ for any minimal generator $\bar{u}$ of $\bar{Z}=Z/x^nZ$.
\end{enumerate}
 Then there exists a minimal generator $e$ of $E$ such that $\height_R O_E(e)\geq k$.
\end{thm}

\begin{proof}
By Lemma \ref{techlemma}, there is a map $\bar{f}:E/x^nE\ra Z/x^nZ$ with $Im(\bar{f})\not \subseteq m(Z/x^nZ)$. Therefore it is possible to pick a minimal generator $\bar{e}$ of $\bar{E}$ such that $\bar{u}=\bar{f}(\bar{e})$ is still a minimal generator of $\bar{Z}$. Thus $O_{\bar{Z}}(\bar{u})\subseteq O_{\bar{E}}(\bar{e})$, yielding $\height_{\bar{R}} O_{\bar{Z}}(\bar{u})\leq \height_{\bar{R}} O_{\bar{E}}(\bar{e})$. The inequality $\height_R O_E(e)\geq min(k,\height_{\bar{R}} O_{\bar{E}}(\bar{e}))\geq min(k,\height_{\bar{R}} O_{\bar{Z}}(\bar{u})) $ follows now from Lemma \ref{EG2lemma}. By the hypothesis, $\height_{\bar{R}} O_{\bar{Z}}(\bar{u}) \geq k$, hence $\height_R O_E(e)\geq k$.
\end{proof}

In the applications detailed in the next section this theorem will be used for a Noetherian local ring $R$ of mixed characteristic $p$ by setting $x=p$.

\subsection{Applications to ranks of syzygies}

Our strategy here will be to use the Weak Order Ideal Theorem and the fact that $k$ is a lower bound on the height of order ideals of minimal generators of $k^{th}$ syzygies in characteristic $p$ to infer the desired lower bound in mixed characteristic.

Since the first Weak Order Ideal Theorem is concerned with elements that annihilate $\Ext$ functors, we begin by showing the inductive  behavior of this property. 

\begin{lem}\label{AnnExt}
 Let $(R,m)$ be a local ring, let $M$ be a finitely generated $R$-module and let $x\in m$ be a non-zerodivisor on $R$. If $x\Ext_R^{k+1}(M,\cdot)\equiv 0$ for a fixed $k>0$, then $x\Ext_R^{j+1}(M,\cdot)\equiv 0$ for all $j\geq k$.
\end{lem}

\begin{proof}
 If $E$ is a $k^{th}$ syzygy for $M$, we note that $\Ext_R^{k+1}(M,\cdot)\simeq \Ext_R^{1}(E,\cdot)$. Since $k>0$, one has that $x$ is regular on $E$ and, further, since $x\Ext_R^{1}(E,\cdot)\equiv 0$ one obtains a pullback diagram in which $F$ is $R$-free and $Z=\Syz_{k+1}(M)$:
$$
\xymatrixrowsep{20pt}
\xymatrixcolsep{40pt}
\xymatrix{
 0 \ar[r] & Z \ar[r]\ar@{=}[d] & Z\oplus E \ar[r]\ar[d] & E \ar[r]\ar[d]^{\cdot x} & 0\\
 0 \ar[r]& Z \ar[r]		& F \ar[r]\ar@{->>}[d] & E \ar[r]\ar@{->>}[d]&	0 \\
         &                       & \bar{E} \ar@{=}[r] & \bar{E} & \\
}
$$
Homological dimension shifting gives $\Ext_R^i(Z\oplus E,\cdot)\simeq \Ext_R^{i+1}(\bar{E},\cdot)$ for $i>0$, thus $x\Ext_R^i(Z\oplus E,\cdot) \equiv 0$ for $i>0$ since $x\bar{E}=0$. Our conclusion follows directly from this assertion and $xExt_R^1(E,\cdot)\equiv 0$.
\end{proof}

\begin{thm}\label{AnnSyz} Let $(R,m)$ be a  Cohen-Macaulay local ring. Fix an integer $k>0$ and assume a superficial element $x$ of $m$ exists with respect to all $j^{th}$ $R$-syzygies with $j\geq k$. If every minimal generator of a $k^{th}$ syzygy over $\bar{R}=R/(x)$ has order ideal of height at least $k$ and  if $M$ is an $R$-module such that $x\Ext_R^{k+1}(M,\cdot)\equiv 0$, then the Syzygy Theorem holds for every $j^{th}$ syzygy of $M$ with $j\geq k$.
 \end{thm}
\begin{proof}
 Let $E$ be the $j^{th}$ syzygy of $M$ with $j\geq k$. By the previous Lemma, $x\Ext_R^{j+1}(M,\cdot)\equiv 0$ so that $x\Ext_R^1(E,Z)=0$ where $0\ra Z\ra F\ra E\ra 0$ is exact with $F$ free. An application of the Weak Order Ideal Theorem \ref{comp} and its consequences in establishing the Syzygy Theorem now yields the desired conclusion.
\end{proof}

\begin{cor}
 With the notation of the previous theorem, if $xM=0$ then the Syzygy Theorem holds for all syzygies of $M$.
\end{cor}

In fact, in this special setting where $xM=0$ one can obtain a stronger result applying results of J. Shamash \cite{S} (see also Proposition 3.3.5 in Avramov's article \cite{A} for a different proof). Details will be provided in section 5.

For the main application of this section we specialize to the case of cyclic modules $R/Q$ with $Q$ a prime ideal.

\begin{thm}\label{PrimeSyz} Let $(R,m)$ be a Cohen-Macaulay local ring. Assume that for every fixed integer $k>2$ and for every $k^{th}$ $R$-syzygy a superficial element $x$ of $m$  with respect to that syzygy exists and that every minimal generator of a $k^{th}$ syzygy over $\bar{R}=R/(x)$ has order ideal of height at least $k$. Then the Syzygy Theorem holds over $R$ for syzygies of modules of the type $R/Q$ with $Q\in Spec(R)$.
 \end{thm}

\begin{proof}
 Depending on whether $x$ is contained in  $Q$ or not and with notations as in the previous theorem,  we have:
\begin{enumerate}
 \item if $x$ is not contained in $Q$, then $M$ is an $\bar{R}$ module which is a $k^{th}$ syzygy of $R/((x)+Q)$. The desired conclusion is given directly by an application of the Syzygy Theorem over $\bar{R}$.
\item if $x$ is contained in $Q$, then we are in the setting of the previous Theorem \ref{AnnSyz}.
\end{enumerate}
\end{proof}

We further specialize $x$ to be $p$, the mixed characteristic in order to prove the Syzygy Theorem holds in the unramified mixed characteristic setting for syzygies of $R/Q$ with $Q$ a prime ideal. The next theorem and corollary follow verbatim from the general versions stated before.

\begin{thm}\label{corp}
 Let $R$ be an unramified Cohen-Macaulay local ring of mixed characteristic $p$ and let $M$ be a finitely generated module such that $p\Ext_R^{k+1}(M,\cdot)\equiv 0$ for some $k>0$. Then the Syzygy Theorem holds for all $j^{th}$ syzygies of $M$ with $j\geq k$.
\end{thm}

\begin{cor}
 The Syzygy Theorem holds for syzygies of modules of the type $R/Q$ with $R$ a regular local ring in unramified mixed characteristic $p$ and $Q\in Spec(R)$.
\end{cor}

Our final consideration of this section concerns syzygy modules for cyclic modules of the form $R/(a,b,p^s)$. The significance of this class of cyclic modules has been discussed in Theorem \ref{reductionAnnp}. Furthermore, the relevance of three generated ideals in the study of syzygies is well known due to Bruns' result in \cite{B2}, which points out that every finite free resolution over a Cohen-Macaulay ring can be obtained (at least from the third syzygy back) as a resolution of a three-generated ideal. Therefore all the pathology that can be encountered is already present in the three-generated ideal case. When $s=1$ and $R$ is Cohen-Macaulay and unramified at $(p)$, we have given a proof that the Syzygy Theorem holds over $R/(p,a,b)$ inTheorem \ref{corp}. However, under supplementary hypotheses one can make a statement regarding the entire family: 

 \begin{prop}
 Let $R$ be a local ring of mixed characteristic $p$ and consider $a,b\in R$.
\begin{enumerate}
 \item if the Syzygy Theorem holds for all cyclic modules $C$ such that $p^{s-1}C=0$, then the Syzygy Theorem  holds for $R/(p^s,a,b)$.
\item the Syzygy Theorem holds for $R/(p^2,a,b)$ for $R$  Cohen-Macaulay and unramified at $(p)$.
\end{enumerate}
\end{prop}

\begin{proof}
 We consider the short exact sequence
$$0 \ra (p,a,b)/(p^s,a,b)\ra R/(p^s,a,b)\ra R/(p,a,b)\ra 0$$
Forming a free resolution of the middle term by taking the direct sum of free resolutions of the first and third terms, shows that the $k^{th}$ syzygy modules for $R/(p^s,a,b)$ and $(p,a,b)/(p^s,a,b)$ are identical for $k>3$ and differ by a free direct summand for $k=3$. The hypothesis in (1) and the fact that $(p,a,b)/(p^s,a,b)$ is annihilated by $p^{s-1}$ gives the desired conclusion. 

Part (2) is an immediate consequence of (1) and the fact that the Syzygy Theorem holds for $R/(p,a,b)$.
\end{proof}

\section{ Weak Order Ideal Theorem via Mapping cone resolutions}

A second comparison theorem with respect to a hypersurface arises from a Cartan-Eilenberg construction. We shall work under assumptions that are reminiscent of the reductions in  Theorem \ref{reductionAnnp} and Proposition \ref{redfinitepd}. Specifically we consider a local domain $R$ and a non-zerodivisor $x\in R$ and we set $S=R/(x^n)$ and study syzygy modules of $R$-modules $T$ with $x^nT=0$ and $\pd_S T < \infty$. We further recall that in studying the Syzygy Conjecture over $R$ one need only look at $k^{th}$ syzygies with $2<k<pd_R T-2$ since the cases $k=1,2$ are well known and since the case $k=pd_R T-2$ was examined in \cite{DG} Corollary 3.5 with a positive outcome.

\subsection{Cartan-Eilenberg construction}

 Towards establishing the weak order ideal property for $R$-syzygies, we consider two minimal free resolutions of $T$. The first resolution $\mathbf{G}.\rightarrow T$ is taken over the hypersurface ring $S$ and the second one $\mathbf{F}.\rightarrow T$ is an $R$-free  resolution of $T$. We use $K_i$ to denote the syzygy modules of $\mathbf{G}.$ and $Z_i$ to denote the corresponding syzygy modules for $\mathbf{F}.$ Applying $S\otimes_R \cdot$ to $\mathbf{F}.$ yields a four-term exact sequence:
$$
\xymatrixrowsep{20pt}
\xymatrixcolsep{25pt}
\xymatrix{
 0\ar[r]&  T \ar[r]^{\delta\hspace{2em}} & Z_1/x^nZ_1 \ar[r]\ar@{->>}[rd] & F_0/x^nF_0 \ar[r] & T\ar[r] & 0.\\
	&		    &				& K_1 \ar[u] & 		& \\
}
$$
If the resolution $\mathbf{F}.\lra T$ was minimal to begin with, then there is  an inclusion of   $K_1$ in $F_0/x^nF_0$ as in the diagram above. Since $F_i/x^nF_i$ has trivial $S$-homology for $i>1$, we can use the $S$-free resolution and a truncated resolution $\mathbf{G}.\ra K_1$ to build a 
(non-minimal) resolution of $Z_1/x^nZ_1$ via the standard Cartan-Eilenberg construction \cite{CE}. 
$$
\xymatrixrowsep{20pt}
\xymatrixcolsep{25pt}
\xymatrix{
0   \ar[r]     & K_{k-1}\ar@{.}[d]\ar[r]	&L_{k-1}\bigoplus Z_k/x^nZ_k \ar@{.}[d]\ar[r] & K_k \ar@{.}[d]\ar[r] & 0\\
      0 \ar[r] &  G_0 \ar[r]\ar[d] &   G_0\bigoplus G_1 \ar[r] \ar[d]& G_1  \ar[r]\ar[d]& 0\\
   0  \ar[r] & T  \ar[r] &   Z_1/x^nZ_1 \ar[r] & K_1 \ar[r] & 0\\ }
$$

\noindent Thus one obtains the short exact sequence of syzygies 
\begin{equation}\label{eqCE}
 0\lra K_{k-1}\lra L_{k-1}\bigoplus Z_k/x^nZ_k \lra K_k\lra 0,
\end{equation}
where $L_{k-1}$ is a free $S$ module produced as a result of the non-minimality of the resolution of $Z_1/x^nZ_1$ in the middle column.

The importance of (\ref{eqCE}) is related to the induced map $Z_k/x^nZ_k\lra K_k$. Since the Order Ideal Theorem holds for syzygy modules of finite projective dimension over $S$, one may achieve a Weak Order Ideal Theorem for the syzygy $Z_k$ if it can be determined that the naturally induced map $Z_k/x^nZ_k\lra K_k$ is nonzero after tensoring with the residue field, for then it will follow that some minimal generator $\bar{e}$ of $Z_k/x^n Z_k$ (which maps by the induced map to a generator of $K_k$) has order ideal $O_{Z_k/x^nZ_k}(\bar{e})$ of height at least $k$.

From Lemma \ref{EG2lemma} it follows that $ht_R O_{Z_q}(e)\geq k$. The construction of the above short exact sequence (\ref{eqCE}) can be further refined so that we can restrict our attention to the regular hypersurface  ring $R/(x)$.

\begin{lem}\label{ldisc}
 Let $R,x$ and $S$ be as above and let $\bar{R}=R/(x)$. If $T$ is an $S$-module such that $pd_S T< \infty$, then $T/xT\simeq (0:_T x)$ and the minimal $S$-free resolution $\mathbf{G}.\lra T$ stays exact after tensoring with $\bar{R}$, so
$$pd_R T-1=pd_s T=pd _{\bar{R}}\bar{T}.$$
\end{lem}
\begin{proof}
 The four-term exact sequence
$$0\lra \bar{R}\lra S \stackrel{x}\lra S\lra \bar{R}\lra 0$$
demonstrates that $\bar{R}$ is a $k^{th}$ syzygy over $S$ for arbitrary large $k$. It follows that $Tor_j^S(\bar{R},T)\simeq Tor_j^S(xS,T)=0$ for $j>0$ and in turn that the induced sequence 
$$0\lra \bar{T}\lra T\stackrel{x}\lra T\lra \bar{T}\lra 0$$ is exact. Thus $\bar{T}\simeq (0:_Tx)$ and the statement concerning projective dimensions follows from the Auslander-Buchsbaum formula.
\end{proof}

\begin{thm}(Weak Order Ideal Theorem)\label{2ndComp}
 Let $R$ be a  local ring, $x$ a non-zerodivisor, $S=R/(x^n)$ with $n>1$ and let $T$ be a finite projective dimension $S$-module.  Assume that the Syzygy Theorem holds over $S$. Let  $K_{k-1}$ denote the $k-1^{st}$ syzygy module for $T$ over $\bar{R}=R/(x)$. If $\rank_{\bar{R}}K_{k-1}\leq 2k-1$, then the $k^{th}$ syzygy $Z_k$ of $T$ over $R$ satisfies the Weak Order Ideal Theorem.
\end{thm}

\begin{proof}
 
Reducing the  two minimal resolutions of $T$ ($\mathbf{G}.\rightarrow T$  taken over the hypersurface ring $S$ and  $\mathbf{F}.\rightarrow T$ taken over $R$) mod $p$ one has $\bar{\mathbf{F}.}\ra T/pT$ and $\bar{\mathbf{G}.}\ra T/pT$ with $H_1(\bar{\mathbf{F}.})=T/pT$ and $\bar{\mathbf{G}.}$ acyclic. In the following, let  $K_{k}'=Syz_k^{\bar{R}}(T/pT)$. Similar to (\ref{eqCE}), we have a short exact sequence of $\bar{R}$-modules
$0\lra T/pT\lra \bar{Z_1}\lra K_1'\lra 0,$ and the same mapping cone construction yields a short exact sequence of syzygy modules 
$$0\lra K_{k-1}'\lra L_{k-1}\bigoplus \bar{Z}_k\lra K_k'\lra 0.$$

 Assume $\im(\bar{Z_k}\ra K_k')\subseteq m_RK_k'$, meaning the induced map $L_{k-1}\lra K_k'$ must be surjective. Consequently we obtain a commutative diagram
$$\xymatrixrowsep{20pt}
\xymatrixcolsep{25pt}
\xymatrix{
0 \ar[r] & K_{k+1}'\ar[r]\ar[d] & L_{k-1} \ar[r]\ar[d] & K_k'\ar[r]\ar[d]& 0\\
0 \ar[r] & K_{k-1}'\ar[r]\ar[d] & L_{k-1}\bigoplus \bar{Z_k}\ar[r]\ar[d] & K_k' \ar[r] & 0\\
         & \bar{Z_k} \ar@{=}[r] & \bar{Z_k} & & \\}
$$
where the left hand column is induced by the upper right square and is a short exact sequence. Since $k+1<pd_{\bar{R}}T$ and since the Syzygy Theorem holds over $\bar{R}$, we obtain that $\rank_{\bar{R}}K_{k+1}'\geq k+1$. Also it must be the case that $\rank_{\bar{R}}\bar{Z}_k=\rank_R Z_k\geq k-1$. Thus $\rank_{\bar{R}}K_{k-1}'\geq k+1+k-1=2k$. This contradicts one of our assumptions.

We have thus shown that some minimal generator $\bar{e}$ of $\bar{Z_k}$ must have its image in $K_k'\setminus m_RK_k'$ and from here we conclude that $\bar{Z_k}$ and $Z_k$ have the weak order ideal property.

\end{proof}

\begin{rmk}
 One can use the short exact sequence (\ref{eqCE}) for a minimal $S$-resolution of $T$ to reduce the hypotheses of the theorem to requiring that $\rank_S K_{k-1}\leq 2k-1$, where $K_{k-1}$ stands for a minimal $k-1^{st}$ syzygy of $T$ over $S$.
\end{rmk}

There is a large class of modules $T$ to which the above Theorem applies: by a theorem of Bruns (\cite{EG} Corollary 3.12) if $\bar{R}$ is a Cohen-Macaulay local ring of dimension $n$ then there exist ideals $I$ which have  free $\bar{R}$ resolution of the form 
$$0\ra \bar{R}\ra \bar{R}^{2n-1}\ra \bar{R}^{2n-3}\ra \cdots \ra \bar{R}^5\ra\bar{R}^3\ra\bar{R}\ra \bar{R}/I\ra 0,$$
with $\rank_{\bar{R}}\Syz_j^{\bar{R}}(\bar{R}/I)=j\leq 2j-1$.
Furthermore, if $R$ is complete regular local of mixed characteristic $p$, then the Syzygy Theorem holds over $S=R/(p^n)$.

\begin{cor}
 Under the hypotheses of Theorem \ref{2ndComp}, one has $\rank_R Z_k\geq k$.
\end{cor}

\begin{cor}
If $\left| \beta_{k-1}^S-\beta_k^S\right|<k-1 $, then $\bar{Z_q}$ has the weak order ideal property.
\end{cor}
\begin{proof}
$\left| \beta_{k-1}^S-\beta_k^S\right|$ represents the positive difference between the ranks of the modules $K_{k-1}$ and $K_{k}$ over $S$.
\end{proof}

\section{Strong Syzygy Theorems}

In this section we point out a stronger lower bound for ranks of syzygies of modules over a local ring $(R,m)$ annihilated by an element $x\in m- m^2$. This result was obtained by J. Shamash \footnote{We wish to thank Sankar Dutta for introducing us to Shamash's article.} in \cite{S} and also L. Avramov in \cite{A}. Free of any assumptions on the annihilator of the module, we show that the same strengthened theorem holds for syzygies of weakly liftable modules.

\subsection{Strong Syzygy Theorem via homotopy splitting}
 We begin by casting some known results on resolutions over hypersurface rings of type $R/(x)$, $x\in m - m^2$ in the light of the ranks of syzygies problem.
 
\begin{thm}[Strong Syzygy Theorem]{\label{StrongSyzShamash}}
 Let $(R,m)$ be a local ring and $M$ an $R$-module annihilated by $x\in m- m^2$. Set $\bar{R}=R/x$ and assume that the syzygy theorem holds over $\bar{R}$. Then 
\begin{enumerate}
 \item $0 \ra \Syz_{k-1}^{\bar{R}}(M)\ra \overline{\Syz_k^R(M)}\ra \Syz_k^{\bar{R}}(M)\ra 0$ for $2\leq k\leq pd M-1$,
\item $\beta_k^R(M)=\beta_{k-1}^{\bar{R}}(M)+ \beta_k^{\bar{R}}(M)$ for $2\leq k\leq pd M-1$,
\item $\rank \Syz_k(M)\geq 2k-1$ for $1\leq k\leq pd(M)-3$.
\end{enumerate}
 where the Betti number $\beta_k(M)$ is the rank of the $k^{th}$ free module in a minimal resolution of $M$ (over $R$ or $\bar{R}$ respectively according to the superscript).
\end{thm}

For  (1) and (2) see section 2 of Shamash's paper \cite{S} or Proposition 3.3.5 and the subsequent remarks in Avramov's lecture \cite{A}. Part (3) follows from (1) and the additional assumption that the Syzygy Theorem holds over $\bar{R}$.

\begin{rmk}
 We note that the stronger bound on ranks of syzygies in Theorem \ref{StrongSyzShamash} requires hypotheses that are quite a bit more restrictive than the ones in Theorem \ref{AnnSyz}. Indeed one can easily construct examples of modules $M$ (even cyclic ones) where $x\in m-m^2$ is such that $xM\neq 0$ but $x\Ext^k_R(M,\cdot)\equiv 0$ for some $k>0$. If $R$ is regular local with $\dim R > 3$, let $0\ra F\ra K\ra I\ra 0$ be a Bourbaki (see Theorem \ref{Bourbaki}) exact sequence in which $K$ is a second syzygy of the residue field $R/m$. Then $ht I=2$ since $K$ is not free, so that $x(R/I)\neq 0$ for any $x\in m-m^2$. Moreover, since the syzygies of $R/I$ are the same as the syzygies of the residue field appropriately shifted in homological degree, if follows that $m\Ext_R^k(R/I,\cdot)\equiv 0$ for $k\geq 2$.
\end{rmk}

\subsection{Strong Syzygy Theorem via weak lifting}

In the following we derive the strong bound on ranks of syzygies of the Strong Syzygy Theorem under different hypotheses. 

Let $R\lra S$ be a ring homomorphism and let $M'$ be an $S$-module. An $R$-module $M$ is called a lifting of $M'$ if  $M'=M\otimes_R S$ and  $Tor_i^R(S,M)=0$ for $i\geq 1$. When $S = R /(x )$ where $x$ is a non-zerodivisor in $R$, a situation which will be our main focus, then  the latter condition for lifting simply says that $x$ must be a non-zerodivisor on $M$. In this case we reformulate the definition above as

\begin{defn}
 Let $R$ be a ring, let $x$ be a non-zerodivisor and not a unit and $\bar{R}=R/(x)$. Let $\bar{M}$ be a $\bar{R}$-module of finite type. We call a $R$-module $M$ of finite type a lifting of $\bar{M}$ if 
\begin{enumerate}
 \item $x$ is not a zerodivisor on $M$ and
\item $\bar{M}\simeq M/xM$.
\end{enumerate}
\end{defn}

In \cite{ADS} Auslander, Ding and Solberg introduce the concept of weakly liftable modules. 

\begin{defn}
Let $R \ra S$ be a ring homomorphism. An $S$-module $M$ is said to weakly lift (or be weakly liftable) to $R$ if it is a direct summand of a liftable module. 
\end{defn}

Questions about lifting can be traced back to Grothendieck who formulated the following lifting problem: Suppose that $(R,m)$ is a complete regular local ring and that $x\in m-m^2$ so that $\bar{R}=R/(x)$ is again regular. If $\bar{M}$ is a $\bar{R}$-module of finite type, does $\bar{M}$ lift to $R$? It is known by work of Hochster \cite{H2} that the answer to this problem is negative.  Explicit conditions for cyclic modules to be liftable and weakly liftable respectively can be found in H. Dao's recent article \cite{Dao}.

In deformation theory a lift from $R/I^i$ to $R$ is called an $i^{th}$ infinitesimal deformation. Auslander, Ding and Solberg show that weak lifting is equivalent to lifting to the first infinitesimal deformation.

\begin{prop}[Proposition 3.2 in \cite{ADS}]\label{infinitesimal}
Let $R$ be a ring, let $x$ be a non-zerodivisor and not a unit and $\bar{R}=R/(x)$. Let $M$ be a $\bar{R}$-module of finite type. A necessary condition for  $M$ to weakly lift is the splitting of the following extension:
$$0\lra M \stackrel{\delta}\lra \overline{\Syz_1^R M} \lra \Syz_1^{\bar{R}}M\lra 0$$

\end{prop}

We observe that if weak lifting occurs, the lower bound on the ranks of syzygies given by the Syzygy Theorem can be strengthened.

\begin{prop} [Strong Syzygy Theorem]\label{StrongSyzWLift}
 Let $(R,m)$ be a local ring, $\bar{R}=R/(x)$ and let  $M$ be a weakly liftable $\bar{R}$ module such that the Syzygy Theorem holds for syzygies of $M$. Then 
\begin{enumerate}
 \item the sequence $0 \ra \Syz_{k-1}^{\bar{R}}(M)\ra \overline{\Syz_k^R(M)}\ra \Syz_k^{\bar{R}}(M)\ra 0$ is split exact  for $2\leq k\leq pd M-1$,
\item $\beta_k^R(M)=\beta_{k-1}^{\bar{R}}(M)+ \beta_k^{\bar{R}}(M)$ for $2\leq k\leq pd M-1$, 
\item $\rank \Syz_k(M)\geq 2k-1$ for $2\leq k\leq pd(M)-2$
\end{enumerate}
where the Betti number $\beta_k(M)$ is the rank of the $k^{th}$ free module in a minimal resolution of $M$ (over $R$ or $\bar{R}$ respectively according to the superscript).
\end{prop}

\begin{proof}
 From the previous lemma, $\bar{Z_1}=\bar{M}\oplus \Syz_1^{\bar{R}}(\bar{M})$. A similar relation holds (with $M$ replaced by $Z_{k-1}$) for $Z_k= \Syz_{1}^{R}(Z_{k-1})$  :
$$\bar{Z}_k=\bar{Z}_{k-1}\oplus \Syz_1^{\bar{R}}(\bar{Z}_{k-1})= \bar{Z}_{k-1}\oplus \Syz_k^{\bar{R}}(\bar{M}).$$
It follows that 
$$\rank _R Z_k= \rank _{\bar{R}} \bar{Z}_k = \rank_{\bar{R}} \bar{Z}_{k-1} + \rank_{\bar{R}} \Syz_k^{\bar{R}}(\bar{M}) \geq (k-1)+k =2k-1,$$
where the bounds on the ranks of the two summands stem from the original Syzygy Theorem.
\end{proof}

Note that the hypotheses of \ref{StrongSyzShamash} and \ref{StrongSyzWLift} are incomparable, in particular in \ref{StrongSyzWLift} we do not require the element $x$ to lie outside the square of the maximal ideal.

 \begin{rmk}
  This observation yields a crude obstruction to lifting modules, at least in case that the lifting occurs modulo an element of $m^2$, as $\rank \Syz_k(M)<2k-1$ will guarantee that $M$ does not lift.
 \end{rmk}
 
  The consequences of the  decomposition of syzygy modules under similar hypotheses in the context of Poincare series have been studied in thorough detail by O'Carroll-Popescu \cite{OP}.

\subsection{Strong Syzygy Theorem via a four-term exact sequence  }

Let $(R,m)$ be a regular local ring and let $E$ be a finitely generated, torsion free $R$-module. One can always find $x\in m - \{0\}$ such that $E[x^{-1}]$ is $R[x^{-1}]$-free. This observation is equivalent to the requirement that $E$ contain a free submodule $F$ such that $x^s(E/F)=0$. In this section we consider the impact of the additional requirement $x\in m- m^2$ when $E$ is a $k^{th}$ syzygy module.
 
\begin{prop}
 Let $(R,m)$ be a local ring and suppose that there exists an element $x$ of $m-m^2$ such that the Syzygy Theorem holds over $\bar{R}=R/(x)$. Suppose further that $E$ is a non-free $k^{th}$ syzygy module over $R$ such that $E$ contains a free submodule $F$ with $x^s(E/F)=0$ for $s\gg 0$. If the quotient module $M=E/F$ has the property $M/xM\simeq (0:_M x)$, then $\rank_R E\geq 2k-1$.
\end{prop}
\begin{proof}
 Our assumptions on $M$ yield a four-term exact sequence where $\ \bar{\cdot} \ $ indicates quotient modules modulo $(x)$.
$$0\ra \bar{M}\stackrel{\delta}\ra \bar{F}\ra \bar{E} \ra \bar{M}\ra 0$$
Our strategy is to argue that $\bar{M}$ and $Z$ are, as $\bar{R}$-modules,  $k^{th}$ and $(k-1)^{st}$ syzygies respectively. A consequence of this is
$$\rank_R E=\rank_{\bar{R}} \bar{E} =\rank_{\bar{R}} \bar{M}+\rank_{\bar{R}} Z\geq k+(k-1)=2k-1.$$

To verify the syzygy property for $\bar{M}$ and $Z$, we use the fact that being a $k^{th}$ $R$-syzygy is equivalent to having Serre's property $S_k$. Towards this end, let $P\in Spec(R)$ be such that $x\in P$. In case $E_P$ is a free $R_P$-module (or equivalently $\bar{E}_P$ is a free $\bar{R}_P$-module), one easily analyzes the four-term sequence to see that both $\bar{M}_P$ and $Z_P$ will be $\bar{R}_P$-free, thus both will be Cohen-Macaulay modules over $\bar{R}_P$. Thus it remains to consider the situation in which $\depth_R E_P \geq k$ (so $\depth \bar{E}_P\geq k-1$), while $\depth_{R_P}\bar{M}_P<k$. An application of the depth lemma (\cite{EG}, Lemma 1.1) yields the contradiction.
$$\depth_{\bar{R}_P}\bar{M}_P=1+\depth_{\bar{R}_P} Z_P\geq 1+1+\depth_{\bar{R}_P}\bar{M}_P=2+\depth_{\bar{R}_P}\bar{M}_P$$
when applied to the four-term sequence.
\end{proof}

\begin{cor}
 Let $R$ and $x$ be as in the previous proposition and suppose $E$ is a non-free $k^{th}$ syzygy module over $R$ such that $x^s\Ext^1_R(E,\cdot)\equiv 0, \mbox{ for } s\gg 0.$
If $0\ra Z\ra F\ra E\ra 0$ is short exact with $F$ free then 
$\rank_R(E\oplus Z) \geq 2k-1.$
\end{cor}
\begin{proof}
 From the information given one may construct a pushout diagram with exact rows and columns:
$$
\xymatrixrowsep{20pt}
\xymatrixcolsep{40pt}
\xymatrix{
 0 \ar[r] & Z \ar[r]\ar[d]^{\cdot x^s} & Z\oplus E \ar[r]\ar[d] & E \ar[r]\ar@{=}[d] & 0\\
 0 \ar[r]& Z \ar[r]\ar@{->>}[d]		& F \ar[r]\ar@{->>}[d] & E \ar[r]            &	0 \\
         & Z/x^sZ \ar@{=}[r]           & Z/x^sZ                &                     & \\
}
$$
We apply the preceding proposition to the middle column (here $M=Z/x^sZ$).
\end{proof}

\begin{cor} 
 Let $R$ be a mixed characteristic local ring unramified at $(p)$ and let $E$ be a non-free $k^{th}$ syzygy ($k\geq 1$). If $E$ contains a free submodule $F$ such that $p(E/F)=0$, then $\rank_R(E)\geq 2k-1$.
\end{cor}
\begin{proof}
 Let $M=E/F$. Since $pM=0$, one has $M/pM\simeq (0:_M p)\simeq M$ and the previous result yields the desired conclusion.
\end{proof}

\section{Final remarks on order ideals}

Let $R$ denote either an $\N$-graded ring in which $R_0$ is a DVR with maximal ideal generated by $p$ or a regular local ring of mixed characteristic $p$. Suppose that $E$ represents a $k^{th}$ syzygy over $R$ having finite projective dimension. A well-known reduction employed in section 2 is that if $E$ is a potential counterexample to the Syzygy Theorem then $E[p^{-1}]$ will necessarily be $R[p^{-1}]$-projective (in fact $E[p^{-1}]$ will be $R[p^{-1}]$-free in the graded case (see  \cite{EG2}, Lemma 7)). Having this property leads one to look for minimal generators $e$ such that $p^s\in O_E(e)$ for positive integers $s\gg 0$, since the Syzygy Theorem can be easily proven as a consequence of this fact (see \cite{EG2}, Proposition 5 and Theorem 6).

In the graded case such a conclusion is achieved in \cite{EG2}. If $R=R_0\oplus R_1\oplus\ldots$ is a standard graded Noetherian ring in which $R_0$ is a DVR having maximal ideal generated by $p$, $R_i$ are finitely generated torsion free $R_0$-modules and $E$ is a finitely generated graded $R$-module such that $E[p^{-1}]$ is $R[p^{-1}]$ projective, then $E[p^{-1}]$ is $R[p^{-1}]$ free (\cite{EG2}, Lemma 7) and furthermore some minimal homogeneous generator $e\in E$ has the desired property $p^s\in O_E(e)$ for $s\gg 0$ (\cite{EG2}, Theorem 8). 

In contrast to the graded case, in the regular local ramified case one cannot hope to establish the Syzygy Theorem in this manner. To illustrate, let $R$ be a regular local ring of mixed characteristic $p$ and suppose that $p\in m_R^2$ and $\dim R\geq 2$. Let $E=m_R$. Then $E[p^{-1}]=R[p^{-1}]$, however no minimal generator $e$ will have the property $p^s\in O_E(e)$ for $s\gg 0$. In fact $O_{m_R}(e)=eR$ for every $e\in m_R$ as a result of the fact that all homomorphisms $m_R\ra R$ are given by $R$-multiples of the natural inclusion.

Even in the unramified setting one can construct examples where the order ideals of minimal generators do not contain any power of $p$. Consider $I=(x^2+y^2+p^2,px,py)$ an ideal of the local ring $R=V[[x,y]]$, $V$ a DVR, $p\in m_V$. It is easy to check that $p^3\in m_RI$ ($p^3=p(x^2+y^2+p^2)-px(x+y)-py(y-x)$), but $p^2\notin I$ because such a statement would imply that $p\in(x,y)$, a contradiction. As remarked in the previous paragraph, order ideals of minimal generators must be principal and the computations above show that no power of $p$ can be contained in any such order ideal.

  The state of affairs in the unramified regular local case is left to be considered with respect to the described techniques. For this reason we have assumed throughout that $R$ is regular local of mixed characteristic $p$ such that $\bar{R}=R/(p)$ is again a regular local ring and we have considered what properties can be derived about order ideals in this context.

\bibliographystyle{amsalpha}

\end{document}